\documentclass[english,conference]{IEEEtran}
\usepackage[T1]{fontenc}
\usepackage[latin9]{inputenc}
\usepackage{amsmath}
\usepackage{amsthm}
\usepackage{amssymb}
\usepackage{esint}
\usepackage{flushend}

\usepackage{pgfplots}
  \pgfplotsset{compat=newest}
  \usetikzlibrary{plotmarks}
  \usepackage{grffile}
\usepackage{graphicx}
\makeatletter
\numberwithin{equation}{section}
\numberwithin{figure}{section}
\theoremstyle{plain}
\newtheorem{thm}{\protect\theoremname}

\makeatother

\usepackage{babel}
\providecommand{\theoremname}{Theorem}

\begin{document}

\title{Recovery of Compactly Supported Functions from Spectrogram Measurements via Lifting}


\author{\IEEEauthorblockN{Sami Merhi}
\IEEEauthorblockA{Department of Mathematics\\
  Michigan State University\\
  East Lansing, MI 48824, U.S.A.\\
  Email: merhisam@math.msu.edu}
\and
\IEEEauthorblockN{Aditya Viswanathan}
\IEEEauthorblockA{Department of Mathematics\\
  Michigan State University\\
  East Lansing, MI 48824, U.S.A.\\
  Email: aditya@math.msu.edu}
\and
\IEEEauthorblockN{Mark Iwen}
\IEEEauthorblockA{Dept. of Mathematics and Dept. of Computational \\ 
  Mathematics, Science and Engineering (CMSE)\\
  Michigan State University\\
  East Lansing, MI 48824, U.S.A.\\
Email: markiwen@math.msu.edu}}

\maketitle

\begin{abstract}
A novel phase retrieval method, motivated by ptychographic imaging,
is proposed for the approximate recovery of a compactly supported
specimen function $f:\mathbb{R}\rightarrow\mathbb{C}$ from its continuous
short time Fourier transform (STFT) spectrogram measurements. The
method, partially inspired by the well known \textit{PhaseLift} \cite{candes2013phaselift}
algorithm, is based on a lifted formulation of the infinite dimensional
problem which is then later truncated for the sake of computation.
Numerical experiments demonstrate the promise of the proposed approach. 
\end{abstract}

\section{Introduction}

The problem of signal recovery (up to a global phase) from phaseless
STFT measurements appears in many audio engineering and imaging applications.
Our principal motivation here, however, is ptychographic imaging (see,
e.g., \cite{rodenburg2007transmission,marchesini2016alternating})
in the 1-D setting where a compactly supported specimen, $f:\mathbb{R}\rightarrow\mathbb{C}$,
is scanned by a focused illuminating beam $g:\mathbb{R}\rightarrow\mathbb{C}$
which translates across the specimen in fixed overlapping shifts $l_{1},\dots,l_{K}\in\mathbb{R}$.
At each such shift of the beam (or, equivalently, the specimen) a
phaseless diffraction image is then sampled in bulk by a detector.
Due to the underlying physics the collected measurements are then
modeled as sampled STFT magnitude measurements of $f$ of the form
\begin{equation}
b_{k,j}:=\left|\int_{-\infty}^{\infty}\!\!\!f\left(t\right)g\left(t-l_{k}\right)e^{-2\pi i\omega_{j}t}dt\right|^{2}\label{equ:ContMeasurements}
\end{equation}
for a finite set of $KN$ shift and frequency pairs $(l_{k},\omega_{j})\in\{l_{1},\dots,l_{K}\}\times\{\omega_{1},\dots,\omega_{N}\}$.
Our objective is to approximate $f$ (up to a global phase) using
these $b_{k,j}$ measurements.

There has been a good deal of work on signal recovery from phaseless
STFT measurements in the last couple of years in the \textit{discrete
setting}, where $f$ and $g$ are modeled as vectors ab initio, and
then recovered from discrete STFT magnitude measurements. In this
setting many related recovery techniques have been considered including
iterative methods along the lines of Griffin and Lim \cite{nawab1983signal,sturmel2011signal}
and alternating projections \cite{marchesini2016alternating}, graph
theoretic methods for Gabor frames based on polarization \cite{salanevich2015polarization,pfander2016robust},
and semidefinite relaxation-based methods \cite{jaganathan2016stft},
among others \cite{eldar2015sparse,bendory2016non,iwen2016fast,robustPR}.

Herein we will instead consider the approximate recovery of $f$ (as
a compactly supported function) from samples of its continuous STFT
magnitude measurements $b_{k,j}$ as per \eqref{equ:ContMeasurements}.
Besides perhaps better matching the continuous models considered in
some applications such as ptychography, and allowing one to more naturally
consider approaches that utilize, e.g., irregular sampling, we also
take recent work on phase retrieval in infinite dimensional Hilbert
spaces \cite{thakur2011reconstruction,cahill2016phase,alaifari2016stable}
as motivation for exploring numerical methods to solve this problem.

In particular, the recent work of Daubechies and her collaborators
implies that the stability of such continuous phase retrieval problems
is generally less well behaved than their discrete counterparts \cite{cahill2016phase,alaifari2016stable}.
Specifically, \cite{alaifari2016stable} characterizes a class of functions for which infinite dimensional phase retrieval (up to a single global phase) from Gabor measurements is unstable, and then proposes the reconstruction of these worst-case functions 
up to several local phase multiples as a stable alternative.  We take this initial work on stable infinite dimensional phase retrieval from Gabor measurements as a further motivation to explore new fast numerical techniques for the robust recovery of compactly supported functions from their continuous spectrogram measurements.



\subsection{The Problem Statement and Specifications}

Given a vector of stacked spectrogram samples from \eqref{equ:ContMeasurements},
\begin{equation}
\vec{b}=\left(\begin{array}{c}
b_{1,1},\dots,b_{1,N},b_{2,1},\dots,b_{K,N}
\end{array}\right)^T \in[0,\infty)^{NK},\label{eq:meas_vec}
\end{equation}
%
our goal is to approximately recover a piecewise smooth and compactly supported function $f:\mathbb{R}\rightarrow\mathbb{C}$. Of course $f$ can only be recovered up to certain ambiguities (such as up to a global phase, etc.) which depend not only on $f$, but also the window function $g$ (see, e.g., \cite{alaifari2016stable}). Without loss of generality, we will assume that the support of $f$ is contained in $[-1,1]$. Given our motivation from ptychographic imaging we will, herein at least, primarily consider the unshifted beam function $g$ to also be (approximately) compactly supported within a smaller subset $[-a,a]\subset[-1,1]$. Furthermore, we will also
assume that $g$ is smooth enough that its Fourier transform decays relatively rapidly in frequency space compared to $\hat{f}$. Examples of such $g$ include both suitably scaled Gaussians, as well as compactly
supported $C^{\infty}$ bump functions \cite{johnson2015saddle}.

\subsection{The Proposed Numerical Approach}

The proposed method aims to recover samples from the Fourier transform
of $f$ at frequencies in $\Omega=\{\omega_{1},\dots,\omega_{N}\}$,
giving $\vec{f}\in\mathbb{C}^{N}$ with $f_{j}=\widehat{f}(\omega_{j})$,
from which $\widehat{f}$ can then be approximately recovered
via standard sampling theorems (see, e.g., \cite{strohmer2005implementations}).
The inverse Fourier transform of this approximation of $\widehat{f}$
then provides our approximation of $f$.

Recovery of the samples from $\widehat{f}$, $\vec{f}\in\mathbb{C}^{N}$,
is performed in two steps using techniques from \cite{iwen2016fast,robustPR}
adapted to this continuous setting: first, a truncated lifted linear
system is inverted in order to learn a portion of the rank-one matrix
$\vec{f}\vec{f}^{*}$ from a finite set of STFT spectrogram samples,
then, an eigenvector based angular synchronization method is used
in order to recover $\vec{f}$ from the portion of $\vec{f}\vec{f}^{*}$
computed in the first step. Note that this truncated lifted linear system is both banded and Toeplitz, with band size determined by the decay of $\widehat{g}$. If $g$ is effectively bandlimited to $[-\delta,\delta ]$ the proposed lifting-based algorithm can be implemented to run in $\mathcal{O}\left(\delta N (\log N + \delta^2)\right)$-time, which is essentially FFT-time in $N$ for small $\delta$. 

\section{Our Lifted Formulation}

The following theorem forms the basis of our lifted setup.
\begin{thm}
Suppose $f:\mathbb{R}\rightarrow\mathbb{C}$ is piecewise smooth and
compactly supported in $\left[-1,1\right]$. Let $g\in L^{2}\left(\left[-a,a\right]\right)$
be supported in $\left[-a,a\right]\subset[-1,1]$ for some $a<1$,
with $\left\Vert g\right\Vert _{L^{2}}=1$. Then for all $\omega\in\mathbb{R}$,
\[
\left|\mathcal{F}\left[f\cdot S_{l}g\right]\left(\omega\right)\right|=\frac{1}{2}\left|\sum_{m\in\mathbb{Z}}\!\!e^{-\pi ilm}\hat{f}\left(\frac{m}{2}\right)\hat{g}\left(\frac{m}{2}-\omega\right)\right|
\]
for all shifts $l\in[a-1,1-a]$. \label{equ:Lifted} 
\end{thm}
\begin{proof}

Denote by $S_{l}g$ the right shift of $g$ by $l$. The short-time
Fourier transform (STFT) \cite{Mallat} of $f$ given $g$, at a shift
$l$ and frequency $\omega$, is defined by 
\[
\mathcal{F}\left[f\cdot S_{l}g\right]\left(\omega\right)=\int_{-\infty}^{\infty}\!\!\!f\left(t\right)g\left(t-l\right)e^{-2\pi i\omega t}dt.
\]

The squared magnitude of the Fourier transform above is called a spectrogram
measurement: 
\[
\left|\mathcal{F}\left[f\cdot S_{l}g\right]\left(\omega\right)\right|^{2}=\left|\int_{-\infty}^{\infty}\!\!\!f\left(t\right)g\left(t-l\right)e^{-2\pi i\omega t}dt\right|^{2}=\left|\left\langle f,h\right\rangle \right|^{2}
\]
where $h\left(t\right)=\overline{g\left(t-l\right)}e^{2\pi i\omega t}$.
We calculate 
\begin{eqnarray*}
\hat{h}\left(k\right) & = & \int_{-\infty}^{\infty}\!\!\!h\left(t\right)e^{-2\pi ikt}dt\\
 & = & \int_{-\infty}^{\infty}\!\!\!\overline{g\left(t-l\right)}e^{2\pi i\omega t}e^{-2\pi ikt}dt\\
 & = & \int_{-\infty}^{\infty}\!\!\!\overline{g\left(\tau\right)}e^{2\pi i\omega\left(\tau+l\right)}e^{-2\pi ik\left(\tau+l\right)}d\tau\\
 & = & e^{2\pi il\left(\omega-k\right)}\int_{-\infty}^{\infty}\!\!\!\overline{g\left(\tau\right)}e^{-2\pi i\left(\omega-k\right)\tau}d\tau.
\end{eqnarray*}
By Plancherel's theorem, we have 
\begin{eqnarray*}
\left|\left\langle f,h\right\rangle \right|^{2} & = & \left|\left\langle \hat{f},\hat{h}\right\rangle \right|^{2} 
  = \left|\int_{-\infty}^{\infty}\!\!\!\hat{f}\left(k\right)\overline{\hat{h}\left(k\right)}dk\right|^{2}\\
 & = & \left|\int_{-\infty}^{\infty}\!\!\!\hat{f}\left(k\right)e^{-2\pi il\left(\omega-k\right)}\overline{\mathcal{F}\left[\overline{g\left(\cdot\right)}\right]\left(\omega-k\right)}dk\right|^{2}\\
 & = & \left|\int_{-\infty}^{\infty}\!\!\!\hat{f}\left(k\right)e^{2\pi ilk}\overline{\mathcal{F}\left[\overline{g\left(\cdot\right)}\right]\left(\omega-k\right)}dk\right|^{2}\\
 & = & \left|\int_{-\infty}^{\infty}\!\!\!\hat{f}\left(\omega-\eta\right)e^{-2\pi il\eta}\overline{\mathcal{F}\left[\overline{g\left(\cdot\right)}\right]\left(\eta\right)}d\eta\right|^{2}\\
 & = & \left|\int_{-\infty}^{\infty}\!\!\!\hat{f}\left(\omega-\eta\right)\hat{g}\left(-\eta\right)e^{-2\pi il\eta}d\eta\right|^{2}
\end{eqnarray*}
where in the last equality we have used 
\[
  \overline{\mathcal{F}\left[\overline{g\left(\cdot\right)}\right]\left(\eta\right)}=\hat{g}\left(-\eta\right).
\]

And so, by Shannon's Sampling theorem \cite{stade}, applied to $\hat{f}$,
we see that $\left|\mathcal{F}\left[f\cdot S_{l}g\right]\left(\omega\right)\right|^{2}$
is equal to
\begin{eqnarray*}
 &  & \left|\int_{-\infty}^{\infty}\!\!\!\hat{f}\left(\omega-\eta\right)\hat{g}\left(-\eta\right)e^{-2\pi il\eta}d\eta\right|^{2}\\
 & = & \left|\int_{-\infty}^{\infty}\!\!\!\hat{g}\left(-\eta\right)\sum_{m\in\mathbb{Z}}\!\!\hat{f}\left(\frac{m}{2}\right)\mbox{sinc}\pi\left(m-2\left(\omega-\eta\right)\right)e^{-2\pi il\eta}d\eta\right|^{2}\\
 & = & \left|\sum_{m\in\mathbb{Z}}\!\!\hat{f}\left(\frac{m}{2}\right)\int_{-\infty}^{\infty}\!\!\!\hat{g}\left(-\eta\right)e^{-2\pi il\eta}\mbox{sinc}\pi\left(m-2\left(\omega-\eta\right)\right)d\eta\right|^{2}\\
 & = & \left|\sum_{m\in\mathbb{Z}}\!\!\hat{f}\left(\frac{m}{2}\right)\left[\hat{g}\left(\cdot\right)e^{-2\pi il\left(\cdot\right)}\star\mbox{sinc}\pi\left(m+2\left(\cdot\right)\right)\right]\left(-\omega\right)\right|^{2}
\end{eqnarray*}
where $\star$ denotes convolution.

Recall that $\mathcal{F}\left[f\star g\right]=\hat{f}\hat{g}$ so
that $f\star g=\mathcal{F}^{-1}\left[\hat{f}\hat{g}\right].$ We calculate
the Fourier transform
\[
\mathcal{F}\left[\hat{g}\left(\cdot\right)e^{-2\pi il\left(\cdot\right)}\right]\left(\xi\right)=\hat{\hat{g}}\left(\xi+l\right)=g\left(-l-\xi\right),
\]
and the Fourier transform $\mathcal{F}\left[\mbox{sinc}\pi\left(m+2\left(\cdot\right)\right)\right]\left(\xi\right)$
as 
\begin{equation}
\mathcal{F}\left[\frac{\sin\pi\left(m+2x\right)}{\pi\left(m+2x\right)}\right]\left(\xi\right)%
=\frac{e^{\pi im\xi}}{2}\chi_{\left(-1,1\right)}\left(\xi\right).
\end{equation}
With this, the spectrogram measurements $\left|\mathcal{F}\left[f\cdot S_{l}g\right]\left(\omega\right)\right|^{2}$
are given by 
\begin{eqnarray*}
 &  & \left|\sum_{m\in\mathbb{Z}}\!\!\hat{f}\left(\frac{m}{2}\right)\!\mathcal{F}^{-1}\left[g\left(-l-\left(\cdot\right)\right)\frac{e^{\pi im\left(\cdot\right)}}{2}\chi_{\left(-1,1\right)}\left(\cdot\right)\right]\!\!\left(-\omega\right)\right|^{2}\\
 & = & \!\!\frac{1}{4}\!\left|\sum_{m\in\mathbb{Z}}\!\!\hat{f}\left(\frac{m}{2}\right)\!\int_{-\infty}^{\infty}\!\!\!g\left(-l-x\right)\!e^{\pi imx}\chi_{\left(-1,1\right)}\!\!\left(x\right)e^{-2\pi ix\omega}\!dx\right|^{2}\\
 & = & \!\!\frac{1}{4}\!\left|\sum_{m\in\mathbb{Z}}\!\!\hat{f}\left(\frac{m}{2}\right)\!\int_{-1}^{1}\!\!\!g\left(-l-x\right)e^{\pi imx}e^{-2\pi ix\omega}dx\right|^{2}\\
 & = & \!\!\frac{1}{4}\!\left|\sum_{m\in\mathbb{Z}}\!\!\hat{f}\left(\frac{m}{2}\right)\!\int_{-l+1}^{-l-1}\!\!g\left(u\right)e^{\pi i\left(-l-u\right)\left(m-2\omega\right)}du\right|^{2}\\
 & = & \!\!\frac{1}{4}\!\left|\sum_{m\in\mathbb{Z}}\!\!\hat{f}\left(\frac{m}{2}\right)e^{-\pi il\left(m-2\omega\right)}\!\int_{-l-1}^{-l+1}\!\!\!g\left(u\right)e^{-2\pi iu\left(\frac{m}{2}-\omega\right)}du\right|^{2}.
\end{eqnarray*}

Since $l$ is such that $\left[-l-1,-l+1\right]\cap\left[-a,a\right]=\left[-a,a\right]$,
we have that $\left|\mathcal{F}\left[f\cdot S_{l}g\right]\left(\omega\right)\right|^{2}$
equals
\begin{eqnarray*}
 &  & \frac{1}{4}\left|\sum_{m\in\mathbb{Z}}\!\!\hat{f}\left(\frac{m}{2}\right)e^{-\pi il\left(m-2\omega\right)}\int_{-a}^{a}\!\!\!g\left(u\right)e^{-2\pi iu\left(\frac{m}{2}-\omega\right)}du\right|^{2}\\
 & = & \frac{1}{4}\left|\sum_{m\in\mathbb{Z}}\!\!\hat{f}\left(\frac{m}{2}\right)e^{-\pi il\left(m-2\omega\right)}\int_{-\infty}^{\infty}\!\!\!g\left(u\right)e^{-2\pi iu\left(\frac{m}{2}-\omega\right)}du\right|^{2}\\
 & = & \frac{1}{4}\left|\sum_{m\in\mathbb{Z}}\!\!e^{-\pi ilm}\hat{f}\left(\frac{m}{2}\right)\hat{g}\left(\frac{m}{2}-\omega\right)\right|^{2}.
\end{eqnarray*}
We have now proven the theorem.
\end{proof}

Using Theorem~\ref{equ:Lifted} we may now write 
\[
\left|\mathcal{F}\left[f\cdot S_{l}g\right]\left(\omega\right)\right|^{2}=\frac{1}{4}\sum_{k\in\mathbb{Z}}\sum_{j\in\mathbb{Z}}A_{k}\overline{A_{j}}
\]
where $A_{n}:=e^{-\pi iln}\hat{f}\left(\frac{n}{2}\right)\hat{g}\left(\frac{n}{2}-\omega\right).$

\subsection{Obtaining a Truncated, Finite Lifted Linear System}

If $\hat{g}$ decays quickly we may truncate the sums above for a
given frequency $\omega$ with minimal error. To that end, we pick
the indices $j$ and $k$ so that $\left|\frac{k}{2}-\omega\right|\leq\delta$
and $\left|\frac{j}{2}-\omega\right|\leq\delta$ for some fixed $\delta\in\mathbb{N}$.
If we denote 
\[
S_{\omega}=\left\{ \left(j,k\right)\in\mathbb{Z\times\mathbb{Z}}|\left|k-2\omega\right|\leq2\delta\text{ and }\left|j-2\omega\right|\leq2\delta\right\} ,
\]
then 
\[
\left|\mathcal{F}\left[f\cdot S_{l}g\right]\left(\omega\right)\right|^{2}=\frac{1}{4}\sum_{(j,k)\in S_{\omega}}A_{k}\overline{A_{j}}+error.
\]
We may write 
\[
\sum_{\left|j-2\omega\right|\leq2\delta}\!\!e^{\pi ilj}\overline{\hat{f}\left(\frac{j}{2}\right)}\overline{\hat{g}\left(\frac{j}{2}-\omega\right)}=e^{2\pi il\omega}\vec{X}_{l}^{*}\vec{Y_{\omega}}
\]
where $\vec{X}_{l}\in\mathbb{C}^{4\delta+1}$ and $\vec{Y}_{\omega}\in\mathbb{C}^{4\delta+1}$
are the vectors 
\[
\vec{X}_{l}=\left(\!\!\! \begin{array}{c}
e^{\pi i l \left(2\delta\right)}\hat{g}\left(-\delta\right)\\
e^{\pi i l \left(2\delta-1\right)}\hat{g}\left(\frac{1}{2}-\delta\right)\\
\vdots\\
e^{\pi i l \cdot 0}\hat{g}\left(0\right)\\
\vdots\\
e^{\pi i l \left(1-2\delta\right)}\hat{g}\left(\delta-\frac{1}{2}\right)\\
e^{\pi i l \left(-2\delta\right)}\hat{g}\left(\delta\right)
\end{array}\!\!\! \right),\,\vec{Y}_{\omega}=\left(\!\!\! \begin{array}{c}
\overline{\hat{f}\left(\omega-\delta\right)}\\
\overline{\hat{f}\left(\omega-\delta+\frac{1}{2}\right)}\\
\vdots\\
\overline{\hat{f}\left(\omega\right)}\\
\vdots\\
\overline{\hat{f}\left(\omega+\delta-\frac{1}{2}\right)}\\
\overline{\hat{f}\left(\omega+\delta\right)}
\end{array}\!\!\! \right).
\]

This notation allows us to write our measurements in a lifted form
\begin{eqnarray*}
\left|\mathcal{F}\left[f\cdot S_{l}g\right]\left(\omega\right)\right|^{2} & \approx & \frac{1}{4}\overline{e^{2\pi il\omega}\vec{X}_{l}^{*}\vec{Y}_{\omega}}\cdot e^{2\pi il\omega}\vec{X}_{l}^{*}\vec{Y}_{\omega}\\
 & = & \frac{1}{4}\vec{X}_{l}^{*}\vec{Y}_{\omega}\vec{Y}_{\omega}^{*}\vec{X}_{l}.
\end{eqnarray*}
Here, $\vec{Y}_{\omega}\vec{Y}_{\omega}^{*}$ is the rank-one matrix
\[ \!
\left[\!\!\!\!\begin{array}{ccccc}
\left|\hat{f}\!\left(\omega-\delta\right)\right|^{2}\!\!\! \!& \!\!\cdots\!\!\! & \overline{\hat{f}\!\left(\omega-\delta\right)}\hat{f}\!\left(\omega\right)\!\!\! & \!\!\cdots\!\!\! & \overline{\hat{f}\!\left(\omega-\delta\right)}\hat{f}\!\left(\omega+\delta\right)\!\!\\
\vdots\!\!\! & \!\!\ddots\!\!\! & \vdots\!\!\! & \!\!\vdots\!\!\! & \vdots\!\!\\
\overline{\hat{f}\!\left(\omega\right)}\hat{f}\!\left(\omega-\delta\right)\!\!\! \! & \!\!\cdots\!\!\! & \left|\hat{f}\!\left(\omega\right)\right|^{2}\!\!\! & \!\!\cdots\!\!\! & \overline{\hat{f}\!\left(\omega\right)}\hat{f}\!\left(\omega+\delta\right)\!\!\\
\vdots\!\!\! & \!\!\vdots\!\!\! & \vdots\!\!\! & \!\!\ddots\!\!\! & \vdots\!\!\\
\overline{\hat{f}\!\left(\omega+\delta\right)}\hat{f}\!\left(\omega-\delta\right)\!\! \!& \!\!\cdots\!\!\! & \overline{\hat{f}\!\left(\omega+\delta\right)}\hat{f}\!\left(\omega\right)\!\!\! & \!\!\cdots\!\!\! & \left|\hat{f}\!\left(\omega+\delta\right)\right|^{2}\!\!
\end{array}\!\!\right]\!\!\!.
\]

For each $\vec{X}_{l}\in\mathbb{C}^{4\delta+1}$, rewrite it as 
\[
\vec{X}_{l}=\left(\begin{array}{ccccc}
m_{-\delta}^{l}, & m_{-\delta+\frac{1}{2}}^{l}, & \dots, & m_{\delta-\frac{1}{2}}^{l}, & m_{\delta}^{l}\end{array}\right)^{T}
\]
so that $m_{k}^{l}=e^{-\pi i l \left(2k\right)}\hat{g}\left(k\right)$.
Then construct the Toeplitz matrix $\mathbf{X}_{l}\in\mathbb{C}^{N\times N}$
as 
\[
\left[\begin{array}{cccccccc}
m_{0}^{l} & m_{\frac{1}{2}}^{l} & \cdots & m_{\delta}^{l} & 0 & 0 & \cdots & 0\\
m_{-\frac{1}{2}}^{l} & m_{0}^{l} & \cdots & m_{\delta-\frac{1}{2}}^{l} & m_{\delta}^{l} & 0 & \cdots & 0\\
\vdots & \vdots & \vdots & \vdots & \vdots & \vdots & \vdots & \vdots\\
0 & 0 & \cdots & 0 & m_{-\delta}^{l} & m_{-\delta+\frac{1}{2}}^{l} & \cdots & m_{\frac{1}{2}}^{l}\\
0 & 0 & \cdots & 0 & 0 & m_{-\delta}^{l} & \cdots & m_{0}^{l}
\end{array}\right]
\]
where $N$ is the number of frequencies $\omega$ being considered.
Then we construct the block matrix ${\bf G}\in\mathbb{C}^{NK\times N}$
as 
\[
\mathbf{G}=\left(\begin{array}{c}
\mathbf{X}_{l_{1}}\\
\mathbf{X}_{l_{2}}\\
\vdots\\
\mathbf{X}_{l_{K}}
\end{array}\right)
\]
where $K$ is the number of shifts of the window $g$.

Let $\mathbf{F}\in\mathbb{C}^{N\times N}$ be defined as 
\[
\mathbf{F}_{i,j}=\begin{cases}
\overline{\hat{f}\left(\frac{i-2n-1}{2}\right)}\hat{f}\left(\frac{j-2n-1}{2}\right), & \text{if }\left|i-j\right|\leq2\delta,\\
0, & \text{otherwise,}
\end{cases}
\]
where $n=\frac{N-1}{4}$. Note that ${\bf F}$ is composed of overlapping
segments of the rank-1 matrices $\vec{Y}_{\omega}\vec{Y}_{\omega}^{*}$
for $\omega\in\{-n,\dots,n\}$.
Thus, our measurements can be written as 
\begin{equation}
\vec{b}\approx\mbox{diag}(\mathbf{GFG^{*}}),\label{eq:measurements}
\end{equation}
where $\vec{b}$ is defined in (\ref{eq:meas_vec}). By consistently
vectorizing (\ref{eq:measurements}), we can obtain a simple linear
system which can be inverted to learn $\vec{F}$, a vectorized version
of ${\bf F}$. In particular, we have 
\begin{equation}
\vec{b}\approx\mathbf{M}\vec{F},\label{eq:lin_system}
\end{equation}
where the matrix $\mathbf{M}\in\mathbb{C}^{NK\times N^{2}}$ can be
computed by, e.g., passing the canonical basis elements for $\mathbb{C}^{N\times N}$,
${\bf E}_{ij}$, through (\ref{eq:measurements}). 

We solve the linear system (\ref{eq:lin_system}) as a least squares
problem; experiments have shown that $\mathbf{M}$ is of rank $NK$. 
The process of recovering the Fourier coefficients of $f$ from $\vec{F}$
is known as angular synchronization, and is described in detail in
\cite{robustPR}.

\section{Numerical Results}
\begin{figure}
\resizebox{3.5in}{2.5in}{
\input{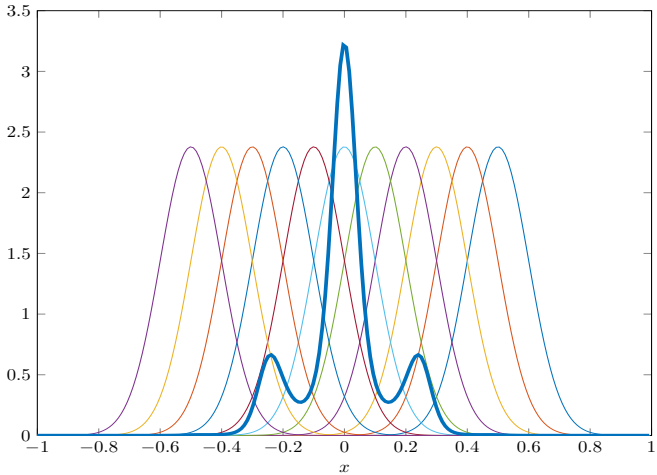}}
\caption{Signal $f$ and 11 shifts of a Gaussian window, $g$. }
\label{fig:functionAndWindows}
\end{figure}

\begin{figure}
\resizebox{3.5in}{2.3in}{
%
\definecolor{mycolor1}{rgb}{0.00000,0.44700,0.74100}%
\definecolor{mycolor2}{rgb}{0.85000,0.32500,0.09800}%
\begin{tikzpicture}

\begin{axis}[%
width=4.521in,
height=3.566in,
at={(0.758in,0.481in)},
scale only axis,
xmin=-1,
xmax=1,
xlabel={$x$},
xmajorgrids,
ymin=-0.2,
ymax=1.2,
ylabel={$f(x)$},
ymajorgrids,
axis background/.style={fill=white},
legend style={legend cell align=left,align=left,draw=white!15!black}
]
\addplot [color=mycolor1,solid,line width=2.0pt]
  table[row sep=crcr]{%
-1	0\\
-0.9755859375	5.06121743032109e-19\\
-0.951171875	4.09467762205061e-18\\
-0.9267578125	3.14177214830437e-17\\
-0.90234375	2.28622993069235e-16\\
-0.8779296875	1.57781244500045e-15\\
-0.853515625	1.03271645914271e-14\\
-0.8291015625	6.4105762557995e-14\\
-0.8046875	3.77401214036781e-13\\
-0.7802734375	2.10717092653665e-12\\
-0.755859375	1.11580033779617e-11\\
-0.7314453125	5.60355271597853e-11\\
-0.70703125	2.6688900695043e-10\\
-0.6826171875	1.20555814037141e-09\\
-0.658203125	5.16459040977684e-09\\
-0.6337890625	2.09833033610546e-08\\
-0.609375	8.08541069660243e-08\\
-0.5849609375	2.95474876896647e-07\\
-0.560546875	1.02406915754675e-06\\
-0.5361328125	3.36610968116037e-06\\
-0.51171875	1.04934304826761e-05\\
-0.4873046875	3.10239410808147e-05\\
-0.462890625	8.69894719443408e-05\\
-0.4384765625	0.000231327184746324\\
-0.4140625	0.000583414072979262\\
-0.3896484375	0.00139545995944246\\
-0.365234375	0.00316554220885597\\
-0.3408203125	0.00681034461031179\\
-0.31640625	0.0138956964390581\\
-0.2919921875	0.0268894425167957\\
-0.267578125	0.0493484514365265\\
-0.2431640625	0.0858925526934506\\
-0.21875	0.141784168759665\\
-0.1943359375	0.221967925503907\\
-0.169921875	0.329566420721337\\
-0.1455078125	0.464072611248535\\
-0.12109375	0.619753760739551\\
-0.0966796875	0.784951193446136\\
-0.072265625	0.942879921912403\\
-0.0478515625	1.07413871939955\\
-0.0234375	1.16052528282121\\
0.0009765625	1.18915671090282\\
0.025390625	1.15561674535778\\
0.0498046875	1.06507161833513\\
0.07421875	0.930966487807927\\
0.0986328125	0.771755146388699\\
0.123046875	0.606757664428955\\
0.1474609375	0.452419440225119\\
0.171875	0.319931863431263\\
0.1962890625	0.214567520899593\\
0.220703125	0.1364773941592\\
0.2451171875	0.0823280275079018\\
0.26953125	0.0471004383841009\\
0.2939453125	0.0255559740639363\\
0.318359375	0.01315073960504\\
0.3427734375	0.00641797718097247\\
0.3671875	0.00297054687000009\\
0.3916015625	0.00130396189399406\\
0.416015625	0.00054285474888351\\
0.4404296875	0.000214334771878567\\
0.46484375	8.0258654575125e-05\\
0.4892578125	2.85023958599493e-05\\
0.513671875	9.59977624289355e-06\\
0.5380859375	3.06641622064437e-06\\
0.5625	9.2894787831241e-07\\
0.5869140625	2.66895862399451e-07\\
0.611328125	7.27248083532313e-08\\
0.6357421875	1.87937558303885e-08\\
0.66015625	4.60611551510044e-09\\
0.6845703125	1.07064704027636e-09\\
0.708984375	2.36019598167956e-10\\
0.7333984375	4.93446488623637e-11\\
0.7578125	9.78413214104483e-12\\
0.7822265625	1.83990248724577e-12\\
0.806640625	3.28138776077242e-13\\
0.8310546875	5.55022451893504e-14\\
0.85546875	8.9033572474687e-15\\
0.8798828125	1.35452585229465e-15\\
0.904296875	1.95438916936272e-16\\
0.9287109375	2.67439188956253e-17\\
0.953125	3.47079766178474e-18\\
0.9775390625	4.27192673961081e-19\\
};
\addlegendentry{True};

\addplot [color=mycolor2,line width=2.0pt,only marks,mark=x,mark options={solid}]
  table[row sep=crcr]{%
-1	-0.000202841191549563\\
-0.9755859375	4.46244013285021e-05\\
-0.951171875	0.000264460889814586\\
-0.9267578125	-4.07829540980006e-05\\
-0.90234375	-0.000352063339999471\\
-0.8779296875	-0.000147934115423612\\
-0.853515625	0.000260659613719319\\
-0.8291015625	0.000405368743992682\\
-0.8046875	0.000123230592145751\\
-0.7802734375	-0.000478404874050352\\
-0.755859375	-0.000680713003406562\\
-0.7314453125	0.000186272871328933\\
-0.70703125	0.00111732219922958\\
-0.6826171875	0.000427409278184326\\
-0.658203125	-0.00113974118288414\\
-0.6337890625	-0.00106685957570456\\
-0.609375	0.000689092018012468\\
-0.5849609375	0.00137184128149167\\
-0.560546875	-2.34147671040072e-05\\
-0.5361328125	-0.00121083274237387\\
-0.51171875	-0.000500741291174009\\
-0.4873046875	0.000757859894921445\\
-0.462890625	0.000778200156433541\\
-0.4384765625	5.25895076111803e-05\\
-0.4140625	0.000115764340711959\\
-0.3896484375	0.00127488063343896\\
-0.365234375	0.00319115529508644\\
-0.3408203125	0.00679026000977555\\
-0.31640625	0.01413682081632\\
-0.2919921875	0.0273208302431449\\
-0.267578125	0.0492955270772962\\
-0.2431640625	0.0852702293126383\\
-0.21875	0.141434637007617\\
-0.1943359375	0.22233681842713\\
-0.169921875	0.330074889751032\\
-0.1455078125	0.464160801523749\\
-0.12109375	0.619576558211796\\
-0.0966796875	0.784785736864478\\
-0.072265625	0.942712453923739\\
-0.0478515625	1.07403021570482\\
-0.0234375	1.16067422420922\\
0.0009765625	1.1894649502137\\
0.025390625	1.1557427893738\\
0.0498046875	1.06495028439369\\
0.07421875	0.930800563298959\\
0.0986328125	0.771586769984151\\
0.123046875	0.606587827330883\\
0.1474609375	0.452543443952691\\
0.171875	0.320458484535641\\
0.1962890625	0.214890639670432\\
0.220703125	0.136075044845245\\
0.2451171875	0.0817264172770004\\
0.26953125	0.0471050316606528\\
0.2939453125	0.0259948171272994\\
0.318359375	0.0133621901616764\\
0.3427734375	0.00639329931942127\\
0.3671875	0.00299914921257817\\
0.3916015625	0.00115431128528479\\
0.416015625	6.46580723738413e-05\\
0.4404296875	9.74187818202472e-05\\
0.46484375	0.000824994246905456\\
0.4892578125	0.000691206803751984\\
0.513671875	-0.000609576459322004\\
0.5380859375	-0.0011835443450301\\
0.5625	0.000120334792290388\\
0.5869140625	0.00140646867125783\\
0.611328125	0.000549618671537344\\
0.6357421875	-0.00116052781734495\\
0.66015625	-0.00104658564174936\\
0.6845703125	0.000545830725635263\\
0.708984375	0.00108894564149215\\
0.7333984375	8.87743255489764e-05\\
0.7578125	-0.000702559543958699\\
0.7822265625	-0.000431290161106109\\
0.806640625	0.000161936523095954\\
0.8310546875	0.000408646176162687\\
0.85546875	0.000233628393027204\\
0.8798828125	-0.0001796084438811\\
0.904296875	-0.000344743664722484\\
0.9287109375	-4.86886245548711e-06\\
0.953125	0.000266260962035537\\
0.9775390625	1.40845556830821e-05\\
};
\addlegendentry{Approx.};

\end{axis}
\end{tikzpicture}%
}
\caption{True signal $f$ and its reconstruction for the first experiment.}
\label{fig:recon_ex1}
\end{figure}
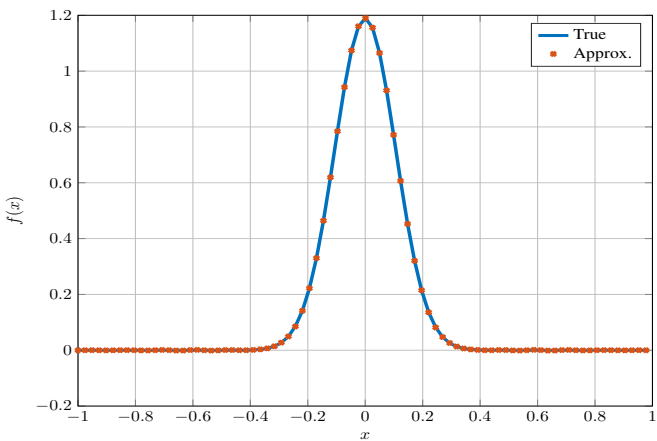

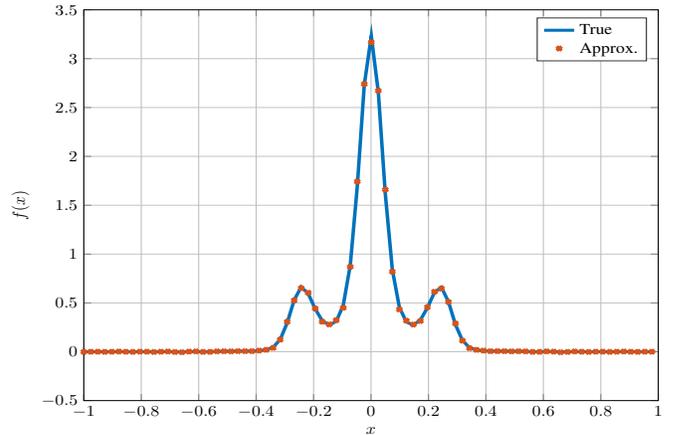
\begin{figure}
\resizebox{3.5in}{2.3in}{
%
\definecolor{mycolor1}{rgb}{0.00000,0.44700,0.74100}%
\definecolor{mycolor2}{rgb}{0.85000,0.32500,0.09800}%
\begin{tikzpicture}

\begin{axis}[%
width=4.521in,
height=3.566in,
at={(0.758in,0.481in)},
scale only axis,
xmin=-1,
xmax=1,
xlabel={$x$},
xmajorgrids,
ymin=-0.5,
ymax=3.5,
ylabel={$f(x)$},
ymajorgrids,
axis background/.style={fill=white},
legend style={legend cell align=left,align=left,draw=white!15!black}
]
\addplot [color=mycolor1,solid,line width=2.0pt]
  table[row sep=crcr]{%
-1	0\\
-0.9755859375	4.19478446377278e-11\\
-0.951171875	8.11503463996002e-11\\
-0.9267578125	1.90498282421332e-10\\
-0.90234375	5.99548609948456e-10\\
-0.8779296875	2.50947298350043e-09\\
-0.853515625	1.24740444452351e-08\\
-0.8291015625	6.14737910004603e-08\\
-0.8046875	2.48959952812876e-07\\
-0.7802734375	7.25975409982051e-07\\
-0.755859375	1.47540513257043e-06\\
-0.7314453125	2.25897890710148e-06\\
-0.70703125	3.06498731677928e-06\\
-0.6826171875	4.46820896961558e-06\\
-0.658203125	8.20263867156594e-06\\
-0.6337890625	2.03739622101898e-05\\
-0.609375	6.5847760484666e-05\\
-0.5849609375	0.000241493061550909\\
-0.560546875	0.000831889122340847\\
-0.5361328125	0.00225245807328559\\
-0.51171875	0.00430384467181926\\
-0.4873046875	0.00579439772800158\\
-0.462890625	0.00610716019089608\\
-0.4384765625	0.00601459058703846\\
-0.4140625	0.00669042635684321\\
-0.3896484375	0.00965928810695736\\
-0.365234375	0.0188765404742947\\
-0.3408203125	0.0466038566306467\\
-0.31640625	0.124242398066006\\
-0.2919921875	0.295002928731489\\
-0.267578125	0.529527302541227\\
-0.2431640625	0.662893946276836\\
-0.21875	0.596146602198795\\
-0.1943359375	0.438587712912628\\
-0.169921875	0.318234484110942\\
-0.1455078125	0.273075945390064\\
-0.12109375	0.31108577875311\\
-0.0966796875	0.475726271026744\\
-0.072265625	0.886198264360338\\
-0.0478515625	1.69156084980378\\
-0.0234375	2.73301519643549\\
0.0009765625	3.23163493519465\\
0.025390625	2.65671400578322\\
0.0498046875	1.61230937863694\\
0.07421875	0.840247694871496\\
0.0986328125	0.455639891712851\\
0.123046875	0.304364648915331\\
0.1474609375	0.273826341958246\\
0.171875	0.325188770588914\\
0.1962890625	0.450898046713292\\
0.220703125	0.60704342529059\\
0.2451171875	0.659789945284359\\
0.26953125	0.511861285703412\\
0.2939453125	0.277644268188253\\
0.318359375	0.115066786637374\\
0.3427734375	0.0431242141278741\\
0.3671875	0.0177101441446335\\
0.3916015625	0.00927421775626059\\
0.416015625	0.00657894624156717\\
0.4404296875	0.00600910697676746\\
0.46484375	0.00611259772103381\\
0.4892578125	0.00572071606574145\\
0.513671875	0.00414250532668777\\
0.5380859375	0.00210509037653878\\
0.5625	0.000758665987991748\\
0.5869140625	0.000217582121392201\\
0.611328125	5.95339801913937e-05\\
0.6357421875	1.87381741534259e-05\\
0.66015625	7.72912680481412e-06\\
0.6845703125	4.30600029607607e-06\\
0.708984375	2.99066023016283e-06\\
0.7333984375	2.19854027730265e-06\\
0.7578125	1.41100492395684e-06\\
0.7822265625	6.7551519833927e-07\\
0.806640625	2.25042074876349e-07\\
0.8310546875	5.43775348057535e-08\\
0.85546875	1.09470378963463e-08\\
0.8798828125	2.21920089789223e-09\\
0.904296875	5.4084367124758e-10\\
0.9287109375	1.76137871021635e-10\\
0.953125	7.65818329685451e-11\\
0.9775390625	3.99273477889542e-11\\
};
\addlegendentry{True};

\addplot [color=mycolor2,line width=2.0pt,only marks,mark=x,mark options={solid}]
  table[row sep=crcr]{%
-1	-0.00128225602157306\\
-0.9755859375	0.000289235931413599\\
-0.951171875	0.000680445687625425\\
-0.9267578125	-0.00126078072238005\\
-0.90234375	0.000363666027880057\\
-0.8779296875	0.00269424891596502\\
-0.853515625	-4.25552288509502e-06\\
-0.8291015625	-0.001612923998171\\
-0.8046875	0.00116519947405553\\
-0.7802734375	0.000485887807705123\\
-0.755859375	-0.00235300809904028\\
-0.7314453125	0.000757715514072846\\
-0.70703125	0.00304716094271794\\
-0.6826171875	-0.00260712376425337\\
-0.658203125	-0.00475487325637731\\
-0.6337890625	0.00188148832341266\\
-0.609375	0.00372291613129194\\
-0.5849609375	-0.00163720112637976\\
-0.560546875	-0.00119937590401694\\
-0.5361328125	0.0043162780479772\\
-0.51171875	0.00499350611977\\
-0.4873046875	0.00403521696846639\\
-0.462890625	0.00601501076672093\\
-0.4384765625	0.00600000396169315\\
-0.4140625	0.00605950597267569\\
-0.3896484375	0.0124079849342762\\
-0.365234375	0.0199579690033015\\
-0.3408203125	0.0398121731112737\\
-0.31640625	0.123950781063476\\
-0.2919921875	0.30655451916598\\
-0.267578125	0.525692952712807\\
-0.2431640625	0.650938464893599\\
-0.21875	0.604190309965675\\
-0.1943359375	0.44260688318313\\
-0.169921875	0.307870957243638\\
-0.1455078125	0.280166246991719\\
-0.12109375	0.323342436399337\\
-0.0966796875	0.450442321535474\\
-0.072265625	0.870575402003276\\
-0.0478515625	1.74169068594273\\
-0.0234375	2.74032689050706\\
0.0009765625	3.16818335301826\\
0.025390625	2.67221401159123\\
0.0498046875	1.65925228246198\\
0.07421875	0.82024235739731\\
0.0986328125	0.433192882404205\\
0.123046875	0.318202649086878\\
0.1474609375	0.278936017233902\\
0.171875	0.314892854218355\\
0.1962890625	0.45643599230653\\
0.220703125	0.613776223568264\\
0.2451171875	0.647001232790648\\
0.26953125	0.509869043787254\\
0.2939453125	0.289093146433415\\
0.318359375	0.113557661616879\\
0.3427734375	0.0366985233820079\\
0.3671875	0.019332062852443\\
0.3916015625	0.0117577423069555\\
0.416015625	0.00583660093142806\\
0.4404296875	0.00611348858192297\\
0.46484375	0.00585961250025348\\
0.4892578125	0.0039831364871015\\
0.513671875	0.00512202031869868\\
0.5380859375	0.00398148682387491\\
0.5625	-0.00156791348643252\\
0.5869140625	-0.00127305256121638\\
0.611328125	0.00397538199114276\\
0.6357421875	0.00132572412561162\\
0.66015625	-0.00499948104954694\\
0.6845703125	-0.00206805617706797\\
0.708984375	0.00318775699021314\\
0.7333984375	0.000353362546378057\\
0.7578125	-0.00230257234616947\\
0.7822265625	0.000747526105459097\\
0.806640625	0.0009663114840813\\
0.8310546875	-0.00170425955539033\\
0.85546875	0.000304188975390063\\
0.8798828125	0.00268693412512695\\
0.904296875	9.11555018466944e-05\\
0.9287109375	-0.00117898682020554\\
0.953125	0.000798514226067969\\
0.9775390625	0.00011796051635268\\
};
\addlegendentry{Approx.};

\end{axis}
\end{tikzpicture}%
}
\caption{True signal $f$ and its reconstruction for the second experiment.}
\label{fig:recon_ex2}
\end{figure}

We test the Phase Retrieval algorithm above for two different choices
of signal $f$. The first is a Gaussian signal $f\left(x\right)=2^{\frac{1}{4}}e^{-25\left(\frac{4x}{3}\right)^{2}}\chi_{\left[-1,1\right]}$,
and the second is a modified Gaussian $f\left(x\right)=2^{\frac{1}{4}}e^{-8\pi x^{2}}\cos\left(24x\right)\chi_{\left[-1,1\right]}$.
In both cases, the window used is the Gaussian $g\left(x\right)=c\cdot2^{\frac{1}{4}}e^{-16\pi x^{2}}\chi_{\left[-\frac{1}{2},\frac{1}{2}\right]}$
where $c$ is a constant chosen so that $\left\Vert g\right\Vert _{L^{2}}=1$.

We use a total of 11 shifts of $g$ in each experiment. Since $g$
is supported on $\left[-\frac{1}{2},\frac{1}{2}\right]$, any two
consecutive shifts are separated by $\frac{0.5}{11}$ (see Figure
~\ref{fig:functionAndWindows}). We choose 61 values of $\omega$
from $\left[-15,15\right]$ sampled in half-steps, and set $\delta=7$.

The reconstructions in physical space are shown at selected grid points
in Figures ~\ref{fig:recon_ex1} and ~\ref{fig:recon_ex2}. The
relative $\ell^{2}$ error in physical space is $1.47\times10^{-3}$
for the first experiment and $1.872\times10^{-2}$ for the second.

\section{Future Work}
While this paper addresses the 1D problem, the extension of this method to the 2D setting is an appealing avenue for future research. Indeed, preliminary results indicate that the underlying discrete method that forms the basis for this paper extends to two dimensions without too much difficulty. 
Furthermore, empirical results suggest that the method proposed here demonstrates robustness to noise, although 
we defer a detailed analysis (and derivation of an associated robust recovery guarantee) to future work.
\section*{Acknowledgement}
This work was supported in part by the National Science Foundation grant 
NSF DMS-1416752.

\bibliographystyle{abbrv}
\bibliography{continuousPRone_IEEE}

\end{document}